\documentclass[reqno,10pt]{amsproc}

\usepackage[top=3cm, bottom=3cm, left=2.54cm, right=2.54cm]{geometry}
\usepackage{amsmath,amsthm,amscd,amssymb,bbm,mathrsfs,latexsym}
\usepackage[colorlinks,citecolor=blue,pagebackref,hypertexnames=false]{hyperref}

\newcommand{\mR}{\mathbb{R}^{3} \times \mathbb{R}^{3}}

\numberwithin{equation}{section}
\theoremstyle{plain}
\newtheorem{theorem}{{\bf Theorem}}[section]

\newtheorem{corollary}[theorem]{Corollary}

\theoremstyle{definition}
\newtheorem{definition}[theorem]{{\bf Definition}}

\theoremstyle{remark}

\numberwithin{equation}{section}

\begin{document}

\title{A direct proof that $\ell_{\infty}^{(3)}$ has generalized roundness zero}

\author{Ian Doust}
\address{School of Mathematics and Statistics, University of New South Wales, Sydney, NSW 2052, Australia}
\email{i.doust@unsw.edu.au}
\author{Stephen S\'{a}nchez}
\address{School of Mathematics and Statistics, University of New South Wales, Sydney, NSW 2052, Australia}
\email{stephen.sanchez@unsw.edu.au}
\author{Anthony Weston}
\address{Department of Mathematics and Statistics, Canisius College, Buffalo, NY 14208, USA}
\email{westona@canisius.edu}
\address{Faculty of Education and Arts, Australian Catholic University, North Sydney, NSW 2060, Australia}
\address{Department of Decision Sciences, University of South Africa, PO Box 392, UNISA 0003, South Africa}
\email{westoar@unisa.ac.za}



\begin{abstract}
Metric spaces of generalized roundness zero have interesting non-embedding properties.
For instance, we note that no metric space of generalized roundness zero is isometric
to any metric subspace of any $L_{p}$-space for which $0 < p \leq 2$.
Lennard, Tonge and Weston gave an indirect proof that $\ell_{\infty}^{(3)}$
has generalized roundness zero by appealing to non-trivial isometric embedding theorems of Bretagnolle
Dacunha-Castelle and Krivine, and Misiewicz. In this paper we give a direct proof
that $\ell_{\infty}^{(3)}$ has generalized roundness zero. This provides insight into the combinatorial
geometry of $\ell_{\infty}^{(3)}$ that causes the generalized roundness inequalities to fail.
We complete the paper by noting a characterization of real quasi-normed spaces of generalized roundness zero.
\end{abstract}
\maketitle

\section{Introduction: negative type and generalized roundness}\label{sec:1}

A notion of generalized roundness for metric spaces was introduced by Enflo \cite{En2} (see Definition \ref{NTGRDEF} (c)).
Enflo constructed a separable metric space of generalized roundness zero that is not uniformly
homeomorphic to any metric subspace of any Hilbert space. This showed that Hilbert spaces are not universal uniform
embedding spaces and thereby settled a question of Smirnov. Enflo's application of generalized roundness to the
uniform theory of Banach spaces remains unique and may, indeed, be regarded as an anomaly. The reason for this is
that generalized roundness is, for all intents and purposes, an isometric rather than uniform invariant. Indeed,
Lennard, Tonge and Weston \cite{LTW} have shown that the generalized roundness and supremal $p$-negative type of
any given metric space coincide. Negative type is a well-known classical isometric invariant whose
origin may be traced back to an 1841 paper of Cayley \cite{Cay}. We recall the relevant definitions here.

\begin{definition}\label{NTGRDEF} Let $p \geq 0$ and let $(X,d)$ be a metric space. Then:
\begin{enumerate}
\item[(a)] $(X,d)$ has $p$-{\textit{negative type}} if and only if for all integers $n \geq 2$,
all finite subsets $\{z_{1}, \ldots , z_{n} \} \subseteq X$, and all choices of real numbers $\zeta_{1},
\ldots, \zeta_{n}$ with $\zeta_{1} + \cdots + \zeta_{n} = 0$, we have:

\begin{eqnarray}\label{NT}
\sum\limits_{i,j = 1}^{n} d(z_{i},z_{j})^{p} \zeta_{i} \zeta_{j} & \leq & 0.
\end{eqnarray}

\item[(b)] $p$ is a \textit{generalized roundness exponent} of $(X,d)$ if and only if for all integers $n > 1$,
and all choices of points $x_{1}, \ldots , x_{n}, y_{1}, \ldots , y_{n} \in X$, we have:

\begin{eqnarray}\label{GR}
\sum\limits_{i,j = 1}^{n} \Bigl\{ d(x_{i},x_{j})^{p} + d(y_{i},y_{j})^{p} \Bigl\} & \leq &
2 \sum\limits_{i,j = 1}^{n} d(x_{i},y_{j})^{p}.
\end{eqnarray}

\item[(c)] The \textit{generalized roundness} of $(X,d)$ is defined to be the supremum of the set of all
generalized roundness exponents of $(X,d)$.
\end{enumerate}
\end{definition}

There is a richly developed theory of negative type metrics that has stemmed from classical papers of
Cayley \cite{Cay}, Menger \cite{Me1, Me2, Me3} and Schoenberg \cite{Sc1, Sc2, Sc3}. Recently there
has been intense interest in negative type metrics due to their usefulness in algorithmic settings.
A prime illustration is given by the \textit{Sparsest Cut problem with relaxed demands} in combinatorial
optimization \cite{Lee, Nao}. There are a number of monographs that provide modern, in-depth, treatments
of the theory of negative type metrics, including Berg, Christensen and Ressel \cite{Ber}, Deza and
Laurent \cite{Dez}, and Wells and Williams \cite{Waw}.

We note that some natural embedding problems involve spaces such as $L_{p}$ with $0 < p < 1$. These spaces carry
the natural quasi-norm $\| \cdot \|_{L_{p}}$ together with the corresponding quasi-metric $d(x, y) = \| x - y \|_{L_{p}}$.
The terminology being used here, however, is not universal. By a \textit{quasi-metric} $d$ on a set $X$ we
mean a function $d$ that satisfies the usual conditions for a metric on $X$, save that the triangle inequality
is relaxed in the following manner: there is a constant $K \geq 1$ such that for all $x,y,z \in X$,
\[
d(x,y) \leq K \cdot \{ d(x,z) + d(z,y) \}.
\]
In the case of $L_{p}$ with $0 < p < 1$, the best possible (smallest) such constant $K$ is $2^{(1-p)/p}$.
The concepts of Definition \ref{NTGRDEF} apply equally well in the broader context of quasi-metric spaces.

It is an important result of Schoenberg \cite{Sc1} that a metric space can be isometrically embedded in some
Hilbert space if and only if it has $2$-negative type. The related problem of embedding Banach
spaces linearly and isometrically into $L_{p}$-spaces was raised by L\'{e}vy \cite{Lev} in 1937. In the case $0 < p \leq 2$,
Bretagnolle, Dacunha-Castelle and Krivine \cite[Theorem 2]{BDK} established that a real quasi-normed space $X$ is linearly
isometric to a subspace of some $L_{p}$-space if and only if $X$ has $p$-negative type. This result was applied
in \cite{BDK} to prove that $L_{q}$ embeds linearly and isometrically into $L_{p}$ if $0 < p < q \leq 2$.
However, in practice it is a hard task to determine if a given real quasi-normed space $X$ has $p$-negative type.

In 1938 Schoenberg
\cite{Sc3} raised the problem of determining those $p \in (0,2)$ for which $\ell_{q}^{(n)}$, $2 < q \leq \infty$, has
$p$-negative type. The cases $q = \infty$ and $2 < q < \infty$ were settled, respectively, by Misiewicz \cite{Mi1}
and Koldobsky \cite{Kol}: if $n \geq 3$ and $2 < q \leq \infty$, then $\ell_{q}^{(n)}$ is not linearly isometric to any
subspace of any $L_{p}$-space for which $0 < p \leq 2$. In other words, by \cite[Theorem 2]{BDK}, $\ell_{q}^{(n)}$
does not have $p$-negative type for any $p > 0$ if $n \geq 3$ and $2 < q \leq \infty$. The restriction $n \geq 3$ on the
dimension of $\ell_{q}^{(n)}$ is essential in this setting. It is well-known that every $2$-dimensional normed space is
linearly isometric to a subspace of $L_{1}$. (See, for example, Yost \cite{Yos}.) In particular, every $2$-dimensional
normed space has generalized roundness at least one.

For any $p \geq 0$, Lennard \textit{et al}.\ \cite[Theorem 2.4]{LTW} have shown that conditions (a) and (b) in
Definition \ref{NTGRDEF} are equivalent. Thus, as noted above, the generalized roundness and supremal $p$-negative type of
any given metric space coincide. It therefore follows from the results in \cite{BDK, Mi1, LTW} that $\ell_{\infty}^{(3)}$
has generalized roundness zero (see \cite[Theorem 2.8]{LTW}). Unfortunately, this indirect proof gives no
insight into the combinatorial geometry of $\ell_{\infty}^{(3)}$ that causes the inequalities (\ref{GR}) to fail
whenever $p > 0$. The main purpose of this paper is to rectify this situation by giving a direct proof that $\ell_{\infty}^{(3)}$
has generalized roundness zero.

The combinatorial geometry which forces $\ell_{\infty}^{(3)}$ to have generalized roundness zero
necessarily carries over to any real quasi-normed space $X$ that contains an isometric copy of $\ell_{\infty}^{(3)}$.
For example, Banach spaces such as $X = c_{0}$ or $C[0,1]$ have generalized roundness zero for the same geometric reasons as $\ell_{\infty}^{(3)}$.

In Section \ref{sec:2} we reformulate Definition \ref{NTGRDEF} (b) in terms of regular Borel measures of compact support.
This facilitates our direct proof that $\ell_{\infty}^{(3)}$ has generalized roundness zero. The argument proceeds by the analysis
of an explicit geometric construction. Equivalently, our arguments give an elementary new proof that $\ell_{\infty}^{(3)}$
does not have $p$-negative type for any $p > 0$. Section \ref{sec:4} completes the paper with a characterization of real
quasi-normed spaces of generalized roundness zero.

\section{A measure-theoretic reformulation of generalized roundness}\label{sec:2}

The purpose of this section is to introduce an equivalent formulation of generalized roundness
that is predicated in terms of measures. The equivalence of conditions (1) and (2) in the statement of
the following theorem is due to Lennard \textit{et al}.\ \cite[Theorem 2.2]{LTW}.

\begin{theorem}\label{equivalence}
Let $(X,d)$ be a metric space and suppose that $p \ge 0$. Then the following are equivalent:
\begin{enumerate}
\item $p$ is a generalized roundness exponent of $(X,d)$.

\item For all integers $N \geq 1$, all finite sets $\{ x_{1}, \ldots, x_{N} \} \subseteq X$,
and all collections of non-negative real numbers $m_{1}, \ldots, m_{N}, n_{1}, \ldots, n_{N}$ that
satisfy $m_{1} + \cdots + m_{N} = n_{1} + \cdots + n_{N}$, we have:
\[
\sum\limits_{i,j = 1}^{N} \bigl\{m_{i}m_{j} + n_{i}n_{j}\bigl\} d(x_{i},x_{j})^{p}
 \leq 2 \sum\limits_{i,j =1}^{N} m_{i}n_{j}d(x_{i},x_{j})^{p}.
\]

\item For all regular Borel probability measures of compact support $\mu$ and $\nu$ on $X$, we have:
  \begin{equation}\label{meas-p}
   \iint_{X \times X} d(x,x')^p \, d\mu(x) d\mu(x')
   +  \iint_{X \times X} d(y,y')^p \, d\nu(y) d\nu(y')
             \le 2 \iint_{X \times X} d(x,y)^p \, d\mu(x) d\nu(y)
  \end{equation}
\end{enumerate}
\end{theorem}

\begin{proof} Let $p > 0$.

Suppose there exists a finite set $\{ x_{1}, \ldots, x_{N} \} \subseteq X$,
and corresponding non-negative real numbers $m_{1}, \ldots, m_{N}, n_{1}, \ldots, n_{N}$ (not all zero)
that satisfy $m_{1} + \cdots + m_{N} = n_{1} + \cdots + n_{N}$, such that
\[
\sum\limits_{i,j = 1}^{N} \bigl\{m_{i}m_{j} + n_{i}n_{j}\bigl\} d(x_{i},x_{j})^{p}
> 2 \sum\limits_{i,j =1}^{N} m_{i}n_{j}d(x_{i},x_{j})^{p}.
\]
By normalization, we may assume that $m_{1} + \cdots + m_{N} = 1 = n_{1} + \cdots + n_{N}$. One may then define
probability measures $\mu$ and $\nu$ on the set $\{ x_{1}, \ldots, x_{N} \}$ as follows: $\mu(\{ x_{i} \}) = m_{i}$ and
$\nu(\{ x_{i} \}) = n_{i}$ for all $i$, $1 \leq i \leq N$. This provides a suitable instance of (\ref{meas-p})
failing.

Conversely, suppose that $\mu$ and $\nu$ are measures such that inequality (\ref{meas-p}) fails. Let $X_0$ be a compact set
containing the support of the two measures. For $n > 0$ let $S_n = \{x_j\}_{j=1}^N$ be a set of points so that the
balls $B(x_i,1/n)$ cover $X_0$. For $x \in X_0$, let $\alpha_n(x)$ equal the element of $S_n$ closest to $x$ (where,
in case of a tie, one takes the element with the smallest index). Let $f_n: X_0 \times X_0 \to \mathbb{R}$ be defined by $f_n(x,y)
= d(\alpha_n(x),\alpha_n(y))^p$. Since the map $(x,y) \mapsto d(x,y)^p$ is uniformly continuous on $X_0 \times X_0$
and $\alpha_n$ converges uniformly to the identity on $X_0$, it is easy to check that $f_n(x,y) \to d(x,y)^p$ uniformly
on $X_0 \times X_0$ and hence that the integrals of $f_n$ with respect to the product measures $\mu \times \nu$, $\nu \times \nu$
and $\mu \times \mu$ converge to the corresponding integrals of $d(\cdot,\cdot)^p$.

Since $f_n$ is a simple function, integrals of it are of the form $\sum_{i,j=1}^n c_{ij} d(x_i,x_j)^p$.
Indeed, if we set $m_i = \mu(\{x \,: \, \alpha_n(x) = x_i\})$ and $n_i = \nu(\{x \,: \, \alpha_n(x) = x_i\})$, then
  \begin{multline}\label{f_n}
     2 \iint_{X \times X} f_n(x,y) \, d\mu(x) d\nu(y)
       - \iint_{X \times X} f_n(x,x') \, d\mu(x) d\mu(x')
      -  \iint_{X \times X} f_n(y,y') \, d\nu(y) d\nu(y')  \\
    = 2 \sum_{i,j=1}^N m_i n_j d(x_i,x_j)^p
       -  \sum_{i,j=1}^N m_i m_j d(x_i,x_j)^p
        - \sum_{i,j=1}^N n_i n_j d(x_i,x_j)^p.
  \end{multline}
Since inequality (\ref{meas-p}) fails, for large enough $n$, the left-hand side of (\ref{f_n}) is negative.
\end{proof}

In the formulation of Definition \ref{NTGRDEF} (b) one may assume that the sets $\{ x_{i} \}$ and $\{ y_{i} \}$ are disjoint.
This is due to cancellation of like terms. In the measure setting, this corresponds to the measures having disjoint support.
We note that analysis of negative type and hypermetric inequalities using measures is not novel. See, for example,
Nickolas and Wolf \cite[Theorem 3.2]{NIC}.

\section{A direct proof that $\ell_{\infty}^{(3)}$ has generalized roundness zero}\label{sec:3}

We now give a direct proof that $\ell_{\infty}^{(3)}$ has generalized roundness $0$ on the basis of a
geometric construction.

\begin{theorem}
For all $p > 0$, $p$ is not a generalized roundness exponent of $\ell_{\infty}^{(3)}$.
\end{theorem}

\begin{proof} Let $d$ denote the metric on $\mathbb{R}^{3}$ induced by the norm of $\ell_{\infty}^{(3)}$.
Fix $p > 0$. Fix $L > 2$ and define the sets
  \begin{align*}
   S_\mu &= \{(t,\pm 1,0) \, :\, -L \le t \le L\}, \text{ and} \\
   S_\nu &= \{(t,0,\pm 1) \, :\, -L \le t \le L\}.
  \end{align*}
So each set is made up of a pair of parallel lines of length $2L$. Define the measures $\mu = \mu_L$ and $\nu = \nu_L$
to be one-dimensional Lebesgue measure supported on the corresponding sets $S_\mu$ and $S_\nu$.

Clearly
  \[ \iint_{\mR} d(x,x')^p \, d\mu(x) d\mu(x')
    =  \iint_{\mR}  d(y,y')^p \, d\nu(y) d\nu(y'), \]
and so we just need to show that, for sufficiently large $L$,
  \begin{equation}\label{eq1}
   \iint_{\mR} d(x,y)^p \, d\mu(x) d\mu(y) > \iint_{\mR} d(x,y)^p \, d\mu(x) d\nu(y).
   \end{equation}

Now
  \begin{align*}
  &\iint_{\mR} d(x,y)^p \, d\mu(x) d\mu(y)\\
  &\qquad = \int_{-L}^L \left( \int_{-L}^L \Vert (t,1,0) - (s,1,0) \Vert^p \, ds
                + \int_{-L}^L \Vert (t,1,0) - (s,-1,0) \Vert^p \, ds\right) \,dt \\
  &\qquad\qquad +
     \int_{-L}^L \left( \int_{-L}^L \Vert (t,-1,0) - (s,1,0) \Vert^p \, ds
                + \int_{-L}^L \Vert (t,-1,0) - (s,-1,0) \Vert^p \, ds \right) \,dt \\
  &\qquad= 2 \Bigl( \int_{-L}^L \int_{-L}^L \Vert (t,1,0) - (s,1,0) \Vert^p \, ds\, dt
       +\int_{-L}^L \int_{-L}^L \Vert (t,1,0) - (s,-1,0) \Vert^p \, ds\, dt
       \Bigr) \\
  &\qquad= 2 \Bigl( \int_{-L}^L \int_{-L}^L |t-s|^p \, ds\, dt
       +  \int_{-L}^L \int_{-L}^L \max(|t-s|,2)^p \, ds\, dt \Bigr).
     \end{align*}
For fixed $t$,
  \[ \int_{-L}^L |t-s|^p \, ds = {\frac { \left( t+L \right) ^{p+1}}{p+1}}+{\frac { \left( L-t \right)
^{p+1}}{p+1}}, \]
and so
  \[ T_1 = \int_{-L}^L \int_{-L}^L |t-s|^p \, ds\, dt
   = 8\,{\frac {{2}^{p}{L}^{p+2}}{{p}^{2}+3\,p+2}}.
  \]
The other term needs to be split into pieces. If $-L \le t \le -L+2$, then
  \[  \int_{-L}^L \max(|t-s|,2)^p \, ds
   = {2}^{p} \left( t+2+L \right) +
   {\frac { \left( L-t \right)^{p+1} - {2}^{p+1}}{p+1}},
  \]
and so
  \begin{align*}
\int_{-L}^{-L+2} & \int_{-L}^L  \max(|t-s|,2)^p \, ds\, dt
  = \int_{L-2}^{L}  \int_{-L}^L  \max(|t-s|,2)^p \, ds\, dt \\
   &= - \,{\frac{2^{p+1} \bigl( 2\,{L}^{2} (L-1)^{p}
   -3\,{p}^{2}
   -4\,L (L-1)^{p}
   -{2}\,{L}^{p+2}
   -7\,p
   +2\,( L-1)^{p}
   -2 \bigr)}
   {{p}^{2}+3\,p+2}}.
  \end{align*}
If $-L+2 \le t \le L-2$, then
  \begin{align*}
  \int_{-L}^L \max(|t-s|,2)^p \, ds
   &= \int _{-L}^{t-2}\! \left( t-s \right) ^{p}{ds}+4\cdot{2}^{p}+\int _{t+2}^
{L}\! \left( s-t \right) ^{p}{ds}
 \\
  &= {\frac { \left( t+L \right) ^{p+1}-{2}^{p+1}}{p+1}}+4\cdot{2}^{p}+{\frac
{\left( L-t \right) ^{p+1}-{2}^{p+1}}{p+1}},
  \end{align*}
and so
  \begin{align*}
  \int_{-L+2}^{L-2} & \int_{-L}^L  \max(|t-s|,2)^p \, ds\, dt \\
  &= {\frac{ 2^{p+3} \bigl(
  ( L-1)^{p}{L}^{2}
  + L {p}^{2}
  - 2\,( L-1)^{p} L
  + 2\, L p
  -2\,{p}^{2}
  +(L-1)^{p}
  - 4\,p
  -1 \bigr)}
  {{p}^{2}+3\,p+2}}.
 \end{align*}
Adding the appropriate terms and simplifying
  \[ T_2 = \int_{-L}^{L}  \int_{-L}^L  \max(|t-s|,2)^p \, ds \,dt
   = {\frac{2^{p+2} \bigl(
   {2}\,L{p}^{2}
   +4\,L p
   -{p}^{2}
   +{2}\,{L}^{p+2}
   -p \bigr)}
   {{p}^{2}+3\,p+2}}.
  \]
Combining $T_1$ and $T_2$, we get
 \[  \iint_{\mR} d(x,y)^p \, d\mu(x) d\mu(y)
   = 2\,(T_1 + T_2)
    ={\frac{2^{p+3} \bigl(
   2\,L{p}^{2}
   +4\,Lp
   -{p}^{2}
   +4\,{L}^{p+2}
   -p \bigr)}
   {{p}^{2}+3\,p+2}}.
  \]
We now turn to calculating the right-hand side of (\ref{eq1}). By symmetry
  \[
   \iint_{\mR} d(x,y)^p \, d\mu(x) d\nu(y)
   = 4 \, \Bigl(
    \int_{-L}^{L} \int_{-L}^L \Vert (t,1,0) - (s,0,1) \Vert^p \,ds\,dt
      \Bigr).
  \]
If $-L \le t \le -L+1$, then
  \begin{align*} \int_{-L}^L \Vert (t,1,0) - (s,0,1) \Vert^p \,ds
   &= 1 \cdot (t+1 - (-L)) + \int_{t+1}^L (s-t)^p \, ds \\
   &= t+1+L+{\frac {\left( L-t \right)^{p+1} - 1}{p+1}}.
  \end{align*}
Thus
  \begin{align*}
   &\int_{-L}^{-L+1}  \int_{-L}^L \Vert (t,1,0) - (s,0,1) \Vert^p \,ds\,dt
   = \int_{L-1}^{L} \int_{-L}^L \Vert (t,1,0) - (s,0,1) \Vert^p \,ds\,dt \\
   &\qquad= \frac {{2}^{p+3}{L}^{p+2}
      - 8\, \left( 2\,L-1 \right)^{p}{L}^{2}
       + 8\, \left( 2\,L-1 \right)^{p} L
       +3\,{p}^{2}
       -2\, \left( 2\,L-1 \right)^{p}
       +7\,p
       +2}
       {2({p}^{2}+3\,p+2)}.
  \end{align*}
If $-L+1 \le t \le L-1$, then
  \begin{align*}
   \int_{-L}^L \Vert (t,1,0) - (s,0,1) \Vert^p \,ds
   &= \int _{-L}^{t-1}\! \left( t-s \right)^{p}{ds}
     +2\cdot 1
     +\int _{t+1}^{L}\! \left( s-t \right)^{p}{ds} \\
     &= \frac { \left( t+L \right)^{p+1}-1}{p+1}
        +2
       + \frac {\left( L-t  \right)^{p+1}-1 }{p+1},
    \end{align*}
and so
  \begin{align*}
   &\int_{-L+1}^{L-1} \int_{-L}^L \Vert (t,1,0) - (s,0,1) \Vert^p \,ds\,dt \\
   & \qquad
   = 2\,{\frac {2\,{p}^{2}L
   + 4 \left( 2\,L-1 \right)^{p}{L}^{2}
   - 4 \left( 2\,L-1 \right)^{p}L
   + \left( 2\,L-1 \right)^{p}
   + 4\,p L
   - 1
   - 2\,{p}^{2}
   -4\,p}
   {{p}^{2}+3\,p+2}}.
  \end{align*}
Combining these gives
  \begin{align*}
    \iint_{\mR} d(x,y)^p \, d\mu(x) d\nu(y)
    &= 4 \int_{-L}^L \int_{-L}^L \Vert (t,1,0) - (s,0,1) \Vert^p \,ds\,dt \\
   &= \frac{4 \bigl( 4\,L{p}^{2}+8\,Lp+8\,{2}^{p}{L}^{p+2}-{p}^{2}-p\bigr)}
   {{p}^{2}+3\,p+ 2}.
  \end{align*}
Let
  \[ \Delta(L,p) = \frac{{p}^{2}+3\,p+ 2}{4p} \Bigl(
   \iint_{\mR} d(x,y)^p \, d\mu(x) d\nu(y)
    - \iint_{\mR} d(x,y)^p \, d\mu(x) d\mu(y) \Bigr).
  \]
From the above calculations we see that
$\Delta(L,p) =
   4\,Lp   + {2}^{p+1}p  + 8\,L +   {2}^{p+1}
     - 4\,L{2}^{p}p  -  8\,L{2}^{p} -  p  -  1$.
It remains to show that no matter how small $p$ is, one can choose $L$ so that $\Delta(L,p) < 0$. (Of course, for any fixed $L$,
$\Delta(L,p) \ge 0$ for $p$ near zero.) It suffices to choose $L = 1/p$. To see this note that
  \begin{align*} \Delta(1/p,p) &=
    4 + {2}^{p+1}p  + \frac{8}{p} + {2}^{p+1}
       - 4 \cdot 2^p - \frac{8 \cdot 2^p}{p} - p - 1 \\
  &= p\, \left( {2}^{p+1}-1 \right) + \left(3-{2}^{p+1}\right) +{\frac{8\,(1-{2}^{p})}{p}}.
  \end{align*}
Elementary calculus shows that for $0 < p < 1$,
 \[ p\, \left( {2}^{p+1}-1 \right) < 3,
   \,
      \left(3-{2}^{p+1}\right) < 1, \text{ and}
   \,\,
      \frac{8\,(1-{2}^{p})}{p}  < -8 \ln 2 < -5,
   \]
and hence for any small $p$, $\Delta(1/p,p) < 0$ as required.
\end{proof}

\section{Quasi-normed spaces of generalized roundness zero}\label{sec:4}

In this brief section we return to the theme of isometrically embedding metric spaces into $L_{p}$-spaces.
The theory developed in the papers \cite{Sc2, BDK, LTW} gives rise to the following characterization of
real quasi-normed spaces of generalized roundness zero.

\begin{theorem}\label{qn:zero}
A real quasi-normed space $X$ has generalized roundness zero if and only
if it is not linearly isometric to any subspace of any $L_{p}$-space for which $0 < p \leq 2$.
\end{theorem}

\begin{proof}
$(\Rightarrow)$ Let $X$ be a real quasi-normed space of generalized roundness zero.
It is plain that a metric space of generalized roundness $\wp$ cannot be isometric to any metric space of generalized
roundness $p > \wp$. Thus $X$ is not isometric to any metric space of positive generalized roundness.
The generalized roundness of any metric subspace of any $L_{p}$-space for which $0 < p \leq 2$ is at
least $p$ by \cite[Corollary 2.6]{LTW}. The forward implication is now evident.

$(\Leftarrow)$ We argue the contrapositive. Let $X$ be a real quasi-normed space of positive generalized roundness.
The set of all $p$ for which a given metric space has $p$-negative type is always
an interval of the form $[0, \wp]$ for some $\wp \geq 0$ or $[0, \infty)$ by Schoenberg \cite[Theorem 2]{Sc2}.
So it follows from \cite[Theorem 2.4]{LTW} that $X$ has $p$-negative type for some $p \in (0,2]$.
Thus $X$ is linearly isometric to a subspace of some $L_{p}$-space by \cite[Theorem 2]{BDK}.
\end{proof}

It is worth noting that in the proof of the forward implication of Theorem \ref{qn:zero} the linear
structures of $X$ and $L_{p}$ play no role. We may therefore infer the following corollary from the
argument given above.

\begin{corollary}
If $X$ is a metric space of generalized
roundness zero, then $X$ is not isometric to any metric subspace of any $L_{p}$-space for which $0 < p \leq 2$.
\end{corollary}

\section*{Acknowledgments}

We thank the referee for detailed and thoughtful comments on the preliminary version of this paper.

\bibliographystyle{amsalpha}

\end{document}